\documentclass[12pt]{amsart}

\usepackage{amsmath,amssymb,amsfonts,amsthm}
\usepackage{graphicx}
\usepackage{float}
\usepackage[all]{xy}
\usepackage{amscd,color}
\usepackage{url}
\usepackage{amscd}
\usepackage{texdraw}
\usepackage{subcaption}

\newtheorem{definition}{Definition}
\newtheorem{theorem}{Theorem}
\newtheorem{remark}{Remark}
\newtheorem*{starthm1}{Theorem A}
\newtheorem*{starthm2}{Theorem B}
\newtheorem{lemma}{Lemma}
\newtheorem{corollary}{Corollary}

\newcommand{\la}{\lambda}
\newcommand{\CC}{\mathbb C}
\newcommand{\RR}{\mathbb R}
\newcommand{\calf}{\mathcal F}

\newcommand{\calm}{\mathcal M}
\newcommand{\calpp}{\mathcal{PP}}
\newcommand{\calp}{\mathcal P}
\newcommand{\cals}{\mathcal S}
\newcommand{\ZZ}{\mathbb Z}
\newcommand{\cala}{\mathcal A}
\newcommand{\DD}{\Delta^{*}}
\newcommand{\cale}{\mathcal E}

\begin{document}

\title[Transversality]
{Transversality at Virtual Centers}

\author{Tao Chen}
\address[]{Department of Math matics, Engineering and Computer Science\\
Laguardia Community College, CUNY\\
31-10 Thomson Ave. Long Island City, NY 11101\\
 and\\
 The Ph.D. Program in Mathematics\\
 Graduate Center of the City University of New York;
 New York, NY 10016}
 \email[Chen]{tchen@lagcc.cuny.edu}

\author{Yunping Jiang}
\address[]{Department of Mathematics\\
 Queens College of the City University of New York\\
 Flushing, NY 11367-1597\\
 and\\
 The Ph.D. Program in Mathematics\\
 Graduate Center of the City University of New York;
 New York, NY 10016}
 \email[Jiang]{yunping.jiang@qc.cuny.edu}

\author{Linda Keen}
\address[]{
 The Ph.D. Program in Mathematics\\
 Graduate Center of the City University of New York;
 New York, NY 10016}
 \email[Keen]{linda.keenbrezin@gmail.com}

 \keywords{dynamically natural slide, shell component, virtual center, transversality, holomorphic motion, lift of holomorphic motion, transfer operator}
 
\subjclass[2020]{Primary 37F46, 37F10; Secondary 32A10, 37F44}

\thanks{The work is partially supported by PSC-CUNY awards, a PSC-CUNY enhanced award (66680-54), and Simons Foundation awards (523341 and 942077).}

%\date{} % Remove date unless needed

\begin{abstract}
 We consider the family of transcendental meromorphic functions with finitely many singular values, none of whose asymptotic values is infinity. We prove that transversality holds at any boundary point of a shell component in a dynamically natural slice of the parameter space at which an asymptotic value is a prepole. 
\end{abstract}

\maketitle

\section{Introduction}\label{intro}
We dedicate this paper to the memory of Bob Devaney whose work on transcendental dynamics has been an inspiration to all of us.

In parameter spaces of families of dynamical systems, there are regions in which the dynamics have essentially the same character.  The points not in these regions are called the ''bifurcation'' locus.  A fundamental problem is to understand how one type of dynamical behavior changes as parameters move through the bifurcation locus.   When a path in the parameter space crosses the bifurcation locus at a single point, it is often easy to understand the behavioral change.   For example,  it is often possible to understand how the entropy of the systems change as they undergo such a ``transversal'' path through the bifurcation locus.  It is therefore important to know where transversality can be guaranteed. 

 In this paper,  we consider transcendental meromorphic functions with finitely many singular values, none of whose asymptotic values is infinity.  Because the functions are meromorphic, there are special parameters in the bifurcation locus for which an asymptotic value is a pole.   We prove that the behavior of the dynamics as we move through these special parameters is transversal.

A tool in proving transversality is to find a ``holomorphic motion''  that tracks an attracting cycle the function $f_{\la}$ as $\la$ moves from a basepoint $\la_0$ in the parameter space.   To do this, one introduces a map called the  ``transfer operator'' that sends a tangent vector to the cycle onto its image under $f_{\la}$.

Before we begin, to set the stage, we recall three simple examples of one complex dimensional families in which each function has exactly two asymptotic values. They are examples of what we will define as {\em dynamically natural slices} of parameter space. 

\begin{itemize}
\item The tangent family: 
$$
\calf_{1}=\Big\{ f_{\lambda}(z) =\lambda \tan z, \quad \la\not=0\Big\}.
$$

\begin{figure}[ht]
     \centering
\centerline{\includegraphics[width=8in]{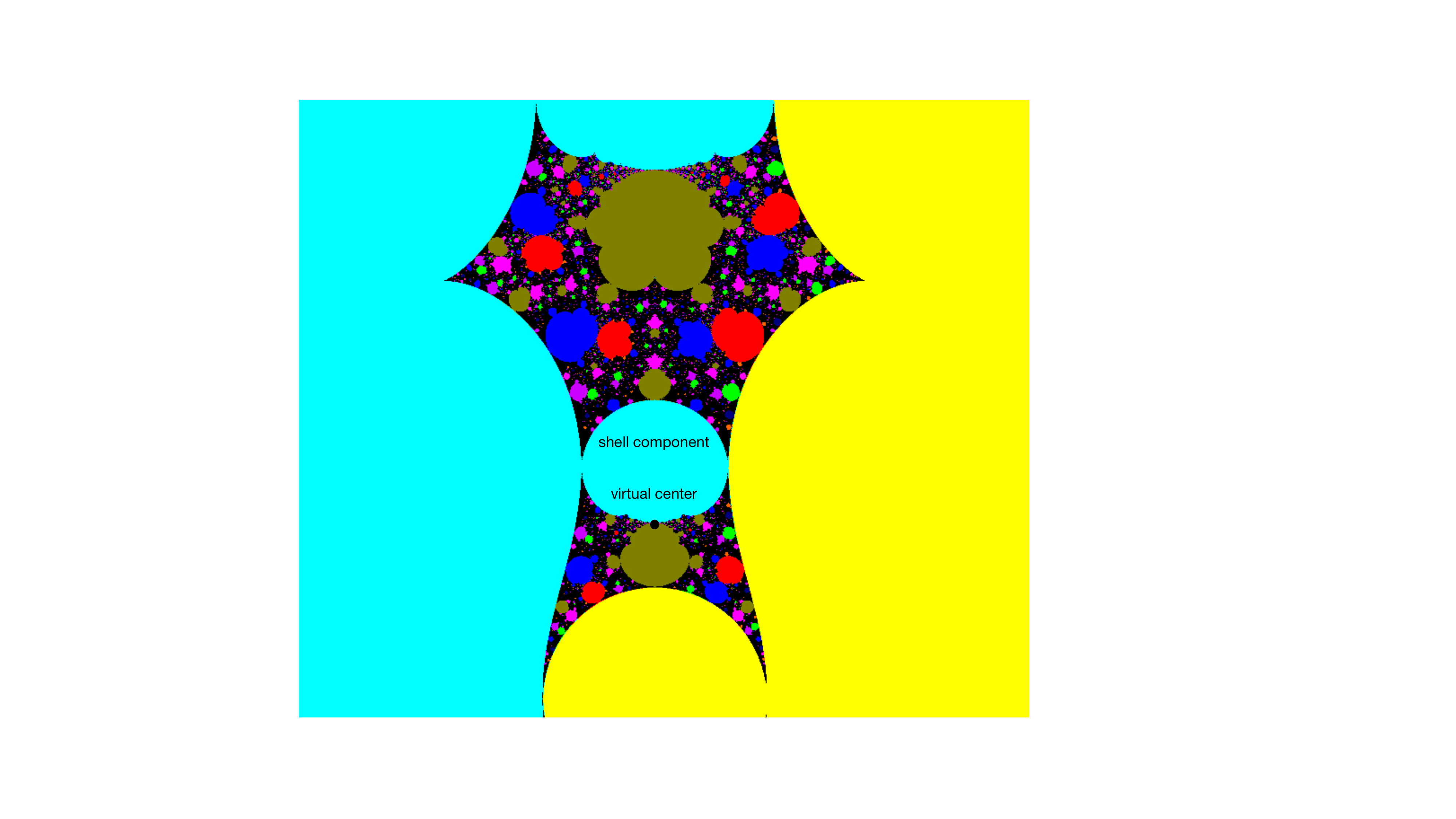}}
  \caption{The tangent family}
  \label{bif1}
\end{figure}

The functions have two symmetric asymptotic values, $\pm \lambda i$.
The bifurcation diagram for this family was studied in~\cite{CJKBif} (see Figure~\ref{bif1} adapted from~\cite{CJKBif}).

\item A family of functions with exactly two asymptotic values such that each function has an attracting fixed point at $0$:
$$
%\begin{equation}~\label{family}
\calf_{2} =\Big\{ f_{\la}(z) = \frac{e^z - e^{-z}}{\frac{e^z}{\la} - \frac{e^{-z}}{\mu}} \mbox{  where  } \frac{1}{\la} - \frac{1}{\mu}=\frac{2}{\rho}, \,0<|\rho|<1\Big\}.
%\end{equation}
$$
The  functions  have two  asymptotic values,  $\la$ and $\mu$,  and at least one of them is attracted to the  origin. The bifurcation diagram for this family was studied in~\cite{CJKNav} (see Figure~\ref{bif2} adapted from~\cite{CJKNav}).

\begin{figure}[ht]
     \centering
 \centerline{\includegraphics[width=8in]{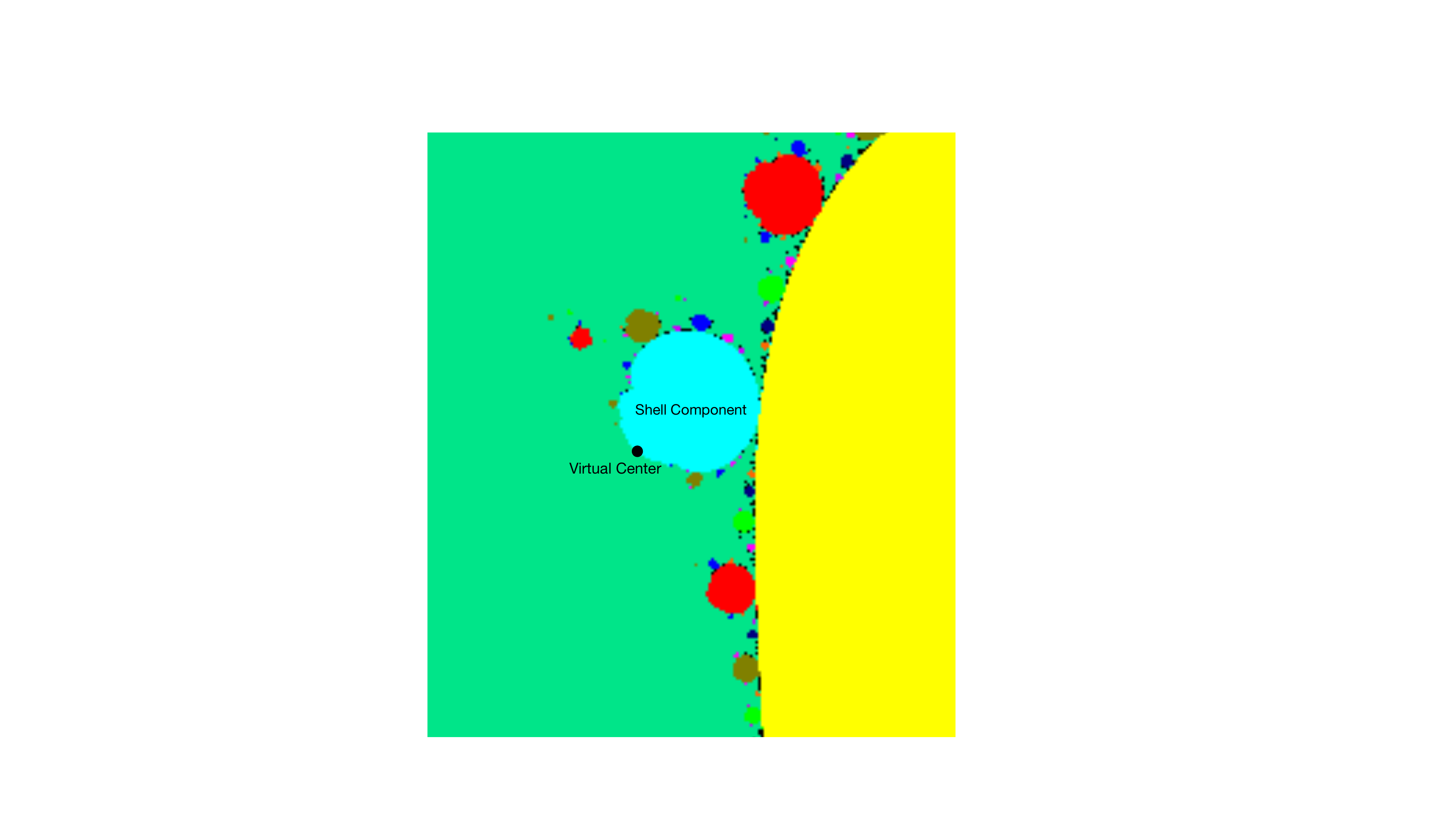}}
  \caption{A family with two asymptotic values}
  \label{bif2}
  %\vspace*{100pt}
\end{figure}

\pagebreak
\item The zero pole family: 
$$
\calf_{3}=\Big\{ f_{\lambda}(z) =\frac{\lambda}{1-e^{-2z}}, \quad \la\not=0\Big\}.
$$ 
The functions have two  asymptotic values, $0$ and $\lambda$; $0$ is always a pole of $f_\la$.
The bifurcation diagram for this family was studied in~\cite{CKZero} (see Figure~\ref{bif3} adapted from~\cite{CKZero}). 

\begin{figure}[ht]
     \centering
 \centerline{ \includegraphics[width=8in]{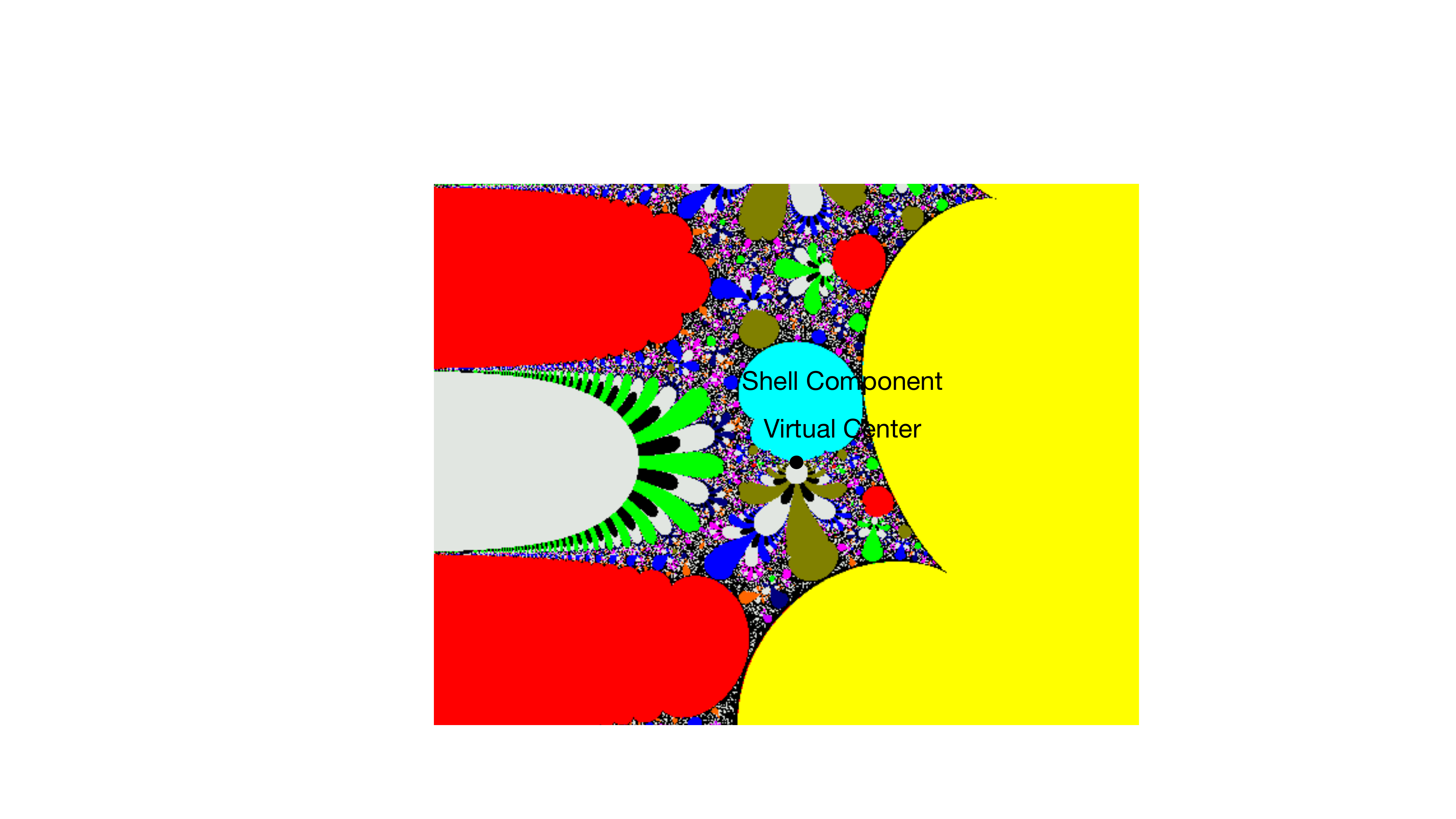}}
  \caption{The zero-pole family}
  \label{bif3}
\end{figure}
\end{itemize}

These figures show the bifurcation loci for the dynamics of three of these slices.  In the colored regions of the figures, the asymptotic values tend to one or two distinct attracting cycles with a non-zero multiplier.  Each region has a special boundary point where the multiplier of the attracting cycle has a limit of zero.   As we discuss below, we call these points ``virtual centers'' of the components.  

Note that all but the central yellow disk in Figure~1 and the green region in Figure~2 look like scallop shells, as do all the components in Figure~3.   These are called {\em shell components}. 

Let $\Lambda$ be  a dynamically natural slice of the parameter space of a meromorphic family of functions $\calm$ whose  singular points are all finite.  We prove that at the virtual centers of $\Lambda$, the transition between dynamically defined regions is locally non-degenerate.  In other words, transversality holds at every virtual center of every shell component of $\Lambda$.   The precise result is Theorem A below.  To get there, we need some notation and definitions.

We use $\lambda$ to denote the parameter in $\Lambda$.  In such a slice, the dynamic behavior of all but one of the singular points is essentially the same.  The dynamics of the remaining singular value $v_{\la}$ vary as $\la$ moves in the slice.  As $\lambda$ runs over a shell component of $\Lambda$, the maps $f_\lambda$ are topologically conjugate on their Fatou sets. There exists a periodic cycle
\[
O_\lambda=\{ z_{\la,0} ,z_{\la,1}, \ldots,z_{\lambda,m-1} \}
\]
that attracts the free asymptotic value $v_\lambda$ and each point in $O(\lambda)$ depends holomorphically on $\lambda$. We label the points of the cycle so that $v_\lambda$ is in the component of the attracting basin containing $z_2(\lambda)$. 
The multiplier
\[
\rho_\lambda=(f_\lambda^{m})'(z_{\lambda,0})
\]
defines a holomorphic universal covering map from the shell component onto the punctured unit disk. This map extends continuously to the boundary of the shell component. 

If one lifts a curve in the punctured disk that lands at the puncture, the lifted curve converges to a boundary point of the shell component. This boundary point is independent of the chosen curve and the multiplier of the corresponding attracting cycle tends to zero.  The point $z_0(\lambda)$ tends to a pole $p_{\la^{*}}$ of $f_{\la^{*}}$ so that the attracting cycle ceases to exist in the usual sense.  
  Nevertheless, the cycle admits a well-defined directional limit,
\[
O_\la^{*}=\{ z_{\la^{*},0}=p_{\la^{*}}, z_{\la^{*},1}=\infty, z_{\la^{*},2}=v_{\la^{*}}, \ldots,z_{\la^*,m-1 }\}.
\]
We call this cycle a {\em virtual cycle} and call $\la^*$  a \emph{virtual cycle parameter $\la^{*}$}.  In analogy with centers of hyperbolic components for rational maps where the corresponding attracting cycle has mulitplier $0$, we also call $\la^{*}$    a \emph{virtual center} of the  component.   We use these terms interchangeably.

We consider the function
$$
%\begin{equation}~\label{eq1}
\Phi(\lambda)=f_\lambda^{m-2}(v_\lambda).
%\end{equation}
$$
The equation 
$$
%\begin{equation}~\label{eq2}
\Phi(\lambda) -p_{\la^{*}}=0
%\end{equation}
$$
has a solution $\la^{*}$. Transversality means that $\la^{*}$ is a simple solution and the function $\Phi (\la)$ is locally univalent throughout the shell component.

Different shell components may meet at a given virtual center. Consequently, virtual cycle parameters are natural bifurcation parameters that mark transitions between regions with distinct dynamics. A fundamental question is whether this transition is locally regular; that is, whether it is transversal. 
%In this paper, we study this question for transcendental meromorphic dynamical systems. 

Our main result concerns the local behavior of the function $\Phi(\la)$ near a virtual center.  We show that, after extending $\Phi(\la)$ appropriately to the virtual cycle parameter, it remains injective and holomorphic in a neighborhood of that point;  that is, transversality holds at every virtual  cycle parameter. More precisely we prove, 

\medskip
\begin{starthm1}~\label{thm1}
 Let $\calf$ be a family of meromorphic functions with finitely many asymptotic values, none of which is infinity;   let $\Lambda= \{ \la \, | f_{\la} \in \calf \}$ be a dynamically natural slice of the parameter space;  denote the free asymptotic value of $f_{\la}$ by $v(\lambda)$.   Let $\Omega$ be a shell component in $\Lambda$ such that the attracting periodic cycle of $f_{\la}, \la \in \Omega$, has period $m$.  If the virtual center of $\Omega$ is $\la^{*}$, denote  the pole in the virtual cycle at $\la^{*}$ by $p_{\la^{*}}$. Then $\Phi'(\la^{*}) \neq 0$;  that is, $\Phi(\la)$ is {\em transversal} at the virtual center $\la^{*}$. 
\end{starthm1}

 The attracting periodic cycle of $f_\lambda$ inside a shell component extends continuously to the virtual cycle of $f_{\lambda^{*}}$ at the virtual center on the boundary. If another shell component (possibly of a different period) is attached at the same virtual center, then this theorem shows that the virtual center serves as a bifurcation parameter between the dynamics of the functions in the components.  That is,  this bifurcation is non-degenerate 
 %and locally monotonic 
 in parameter space. 
%
%In~\cite{CJKBif},  we studied this bifurcation for the tangent family $\lambda \tan z$  where we called it {\em period merging}.  In~\cite{CJKNav} we studied it for the family in example 2 above where we proved every virtual center is accessible both from inside the shell component and from inside the green component where both asymptotic values are attracted to the origin.  In~\cite{CK} we proved  that at  every virtual center infinitely many components meet and transversality holds.  

In order to prove Theorem~A we need to prove another theorem which is of independent interest.
Its  proof relies on the theory of holomorphic motions (see~\cite{GJW}). Let
\[
h(\la,z):\Delta_r\times P\to\widehat{\CC}
\]
be a holomorphic motion of the finite set $P$ consisting of the points of a virtual cycle, where $\Delta_r$ denotes the disk of radius $r>0$ centered at the origin. Using the dynamics of $f_\lambda$, we construct a sequence of lifted holomorphic motions
\[
h_k(\la,z),
\]
each compatible with the action of $f_\lambda$. Associated with these lifts is a natural transfer operator $A$ acting on the corresponding tangent spaces. The spectral properties of this operator play a central role in the proof of transversality.  (See~\cite{T}.)
Before proving Theorem A, we establish the following theorem.

\medskip
\begin{starthm2}~\label{thm2} 
Transversality holds at the virtual center $\la^{*}$ if and only if $1$ is not an eigenvalue of the transfer operator $A$.
 \end{starthm2}

 The paper is organized as follows. In Section 2, we introduce the singular values of a transcendental meromorphic function (Definition~\ref{singlepts}) and the notion of a virtual orbit of a pole (Definition~\ref{extorbit}). In Section 3, we define dynamically natural slices (Definition~\ref{dynnatslice}), shell components (Definition~\ref{shell}), virtual cycles (Definition~\ref{vc}), and virtual cycle parameters (Definition~\ref{vcp}). We also define the virtual centers of shell components and review their basic properties (Theorem~\ref{shellcomponents}).

Section 4 is devoted to the proof of our main result, Theorem A. In Subsection 4.1, we review the theory of holomorphic motions and introduce the notion of a lift of a holomorphic motion through a holomorphic family of meromorphic functions (Definition~\ref{liftdef}). We then define the transfer operator associated with this lifting procedure. In Subsection 4.2, we prove that the spectral radius of the transfer operator is at most $1$ (Lemma~\ref{sp}). In Subsection 4.3, we establish our second result, Theorem B, which shows that transversality is equivalent to the spectral radius of the transfer operator  being strictly less than $1$. Finally, in Subsection 4.5, we complete the proof of Theorem A.

\section{Families of meromorphic functions }\label{basic}
  In this paper, unless we specifically say otherwise, we always assume that an entire or meromorphic map $f$ is transcendental and so has infinite degree.  If we mean a map of finite degree, we call it a polynomial or a rational map.  We also always assume infinity is an essential singularity.   Except for subsection~\ref{dirlimits}, the material on the dynamics of meromorphic functions in this section is standard; see, for example,~\cite{Berg, Mil}.
  We need the following definitions:

\medskip
 \begin{definition}~\label{singlepts} A point $v\in\widehat{\CC}$ is called a {\em singular value}\footnote{In the literature these are sometimes called singularites of $f^{-1}$.} of $f$ if, for some small neighborhood of $v$, some branch of $f^{-1}$  is not well defined. We denote the set of singular points of $f$ by $\cals_f$.
 \end{definition}
 
If the derivative $f'(c)=0$, we call $c$ a critical point, and its image $v=f(c)$ a {\em critical value}. A critical value is singular since the branch with $f^{-1}(v)=c$ is not well defined.  If there is a path $\gamma(t)$ such that $\lim_{t \to \infty} \gamma(t)=\infty$ and $\lim_{t \to \infty} f(\gamma(t))=v$, then $v$ is singular and called an {\em asymptotic value} of $f$; $\gamma (t)$ is called an {\em asymptotic path.} 
Singular values may be critical, asymptotic, or accumulations of such points.    

If  an asymptotic value $v$ is isolated, and if $\gamma$ is an asymptotic path for $v$, we can find a nested sequence of neighborhoods of $v$, $D_{r,v}$ with $r\to 0$,  and a particular branch $g$ of $f^{-1}$ such that $V_{r,v}=g(D_{r,v}\setminus \{v\})$ is a nested sequence of neighborhoods containing segments of $\gamma$ and $\cap g(D_{r,v}\setminus \{v\}) = \emptyset$.  Then $r$ can be chosen small enough so that the map $f: V_{r,v} \rightarrow D_{r,v} \setminus \{v\}$ is a universal covering map.  In this case, $V_v=V_{r,v}$, for any choice of $r$,  is called an {\em asymptotic tract} for the asymptotic value $v$, and $v$ is called a {\em logarithmic singular value}.  The number of distinct asymptotic tracts of a given asymptotic value is called its {\em multiplicity}\footnote{Note this notion is not the same as the multiplicity of a critical point.}.

\medskip
 \begin{definition}\label{Julia} A ray $\beta$ approaching infinity is called a {\em Julia ray} or {\em Julia direction} for the meromorphic function $f$ if $f$ assumes all (but at most two) values infinitely often in any sector containing $\beta$.
 \end{definition}

An example to keep in mind is the tangent function $\tan z$ with asymptotic values $\pm i$, asymptotic tracts the upper and lower half planes, and Julia directions parallel to the positive and negative real axes.

In this paper, whenever we talk about the number of critical points and/or asymptotic values, we tacitly assume that we count with multiplicity.

An entire function $f$ always has an asymptotic value at infinity.  We can,  however,  use a M\"obius transformation to conjugate it to a transcendental meromorphic function.  Transcendental meromorphic functions with $\# \mathcal{S}_f < \infty$ are called functions of {\em finite type}. Because they have infinite degree,  such functions must have asymptotic values. Moreover, all the asymptotic values of a finite-type function are isolated and thus logarithmic.
 In this paper, we focus on the subclass of transcendental meromorphic functions $f: \CC\to  \widehat{\CC}$ of finite type such that infinity is  {\em not } an asymptotic value. This condition implies that the function has infinitely many poles. We denote this class $\calm$.  

\subsection{Dynamics of $f\in \calm$}\label{dynamics.} Let $f^n$ denote the $n^{th}$ iterate of $f$;  that is, $f^n(z)=f(f^{n-1})(z)$ for $n\geq 1$ and $f^{0}=Id$. Then $f^n$ is well-defined, except at  the poles and prepoles of $f$;  this countable set is defined as 
 \[ \calpp(f)= \{z \,| \, f^{n-1}(z) = \infty, \, \mbox{ for some } \, n= 2, 3, \ldots  .\} \]
 The points in $\calpp(f)$ have finite orbits that end at infinity.
 It will be convenient to denote the subset consisting of the poles ($n=2$) of $f$ by $\calp(f)$. If we are working with a fixed function $f$, we will abuse notation and write $\calpp$ and $\calp$.

 \subsection{Directional limits}\label{dirlimits}
  For $f\in \calm$, we can define extensions of the finite orbits of points in $\calpp(f)$ that depend on taking what are essentially directional limits.  Since the asymptotic values are counted by multiplicity, each of the asymptotic values $v$ has at least one and at most finitely many asymptotic tracts, $A_{i}(v)$.

 If $p \in \calp$, and $\gamma_i(t)$ is an asymptotic path in the asymptotic tract $A_i(v)$, let $\delta_{ij}(t)$ be one of connected components of $f^{-1}(\gamma_i(t))$ so that $\lim_{t \to \infty}\delta_{ij}(t)=p$.  
 
 \medskip
\begin{definition}~\label{dirlimit} 
With the above notation, we define the  {\em directional limit} as  
 \[  
 \lim_{t \to \infty}f^2(\delta_{ij}(t))= \lim_{t \to \infty} f(\gamma_i(t) ) = v,  
 \]  
or, if there is no confusion, more simply  as 
 $$
 f^{2} (p) =f(\infty) =v.
 $$
  \end{definition}
  
 \begin{remark} Note that the directional limit depends only on the asymptotic value and not on the choice of asymptotic tract or path. 
  \end{remark}
 In view of Definition~\ref{dirlimit}, the asymptotic values are directional limits of infinity.  This allows us to define what we call virtual orbits for every pole. 
   
 \medskip
 \begin{definition}\label{extorbit} If $v$ is an asymptotic value, the {\em virtual $v$-orbit of a pole $p$} is the set 
 $$
 %\begin{equation} \label{extorbiteqn} 
 \{p, \infty, v, f(v), f^2(v), \ldots \}. 
 $$
% \end{equation}
 \end{definition}

 \subsection{Classification}  The orbits of points under $f\in \calm$ form two complimentary completely invariant sets,  the \emph{Fatou set} and \emph{Julia set} of $f$. The {\em Fatou or stable set} is defined by
\[  F(f) = \{ z\in \CC \ | \ f^n \text{ is defined and normal in a neighborhood of } z\}.\]   It is open, but may be empty. The {\em Julia or unstable set} is its compliment, 
$$
J(f)=\widehat{\CC}\setminus F(f).
$$
Note that the point at infinity is always in the Julia set. Thus, the Julia set is a non-empty compact set in $\widehat{\CC}$.  

Since $\calp$ contains infinitely many points, by Picard's theorem, $\{f^n\}_{n=0}^{\infty}$ is normal on $\widehat{\CC}\setminus\overline{\calpp}$. Since 
$\{f^n\}_{n=0}^{\infty}$ is not normal on $\overline{\calpp}$, 
$$
J(f)=\overline{\calpp}.
$$

A point $z$ is called a periodic point of period $p\geq 1$ if $f^p(z)=z$ and $f^k(z)\neq z$ for any $0<k<p$. For such a point $z$, we have a cycle 
$ O =\{z_0=z, z_1=f(z), \ldots, z_{p-1}=f^{p-1}(z) \}$. The {\em multiplier} of the cycle is  
$$
\rho=\rho(O)=(f^p)'(z).
$$
The cycle $O$ is \emph{attracting} if $0<|\rho|<1$; \emph{super-attracting} if $\rho=0$; \emph{parabolic} (or rational neutral) if $\rho=e^{2\pi i \theta}$ for a rational number $\theta$; {\em irrational neutral} if $\rho=e^{2\pi i \theta}$ for an irrational number $\theta$ and \emph{repelling} if $|\rho|>1$. The Julia set contains all repelling periodic points. 

It was shown in~\cite{FK} that if an asymptotic value $v$ lands on a pole $p$, so that the $v$-virtual orbit of $p$ is periodic,  the directional limit of the multiplier is zero.   Such an orbit is called a {\em virtual cycle} (see Definition~\ref{vc} for a more precise definition). From a multiplier point of view,  a virtual center looks like a super-attracting periodic cycle in the rational case; however, the points in a virtual cycle belong to the Julia set, while the points in a super-attracting cycle belong to the Fatou set. Actually, a virtual center behaves more like a parabolic cycle in the rational case because  we can see a very rich bifurcation phenomena near a virtual center.      

If  $D$ is a component of the Fatou set, then $f(D)$ is either a component of the Fatou set or a component punctured at a single point. For the orbit of $D$ under $f$, there are  two possibilities:
\begin{itemize}
\item[i)]  there exist integers $m\neq n\geq 0$ such that $f^m(D)\subset f^n(D)$, and $D$ is called \emph{eventually periodic};
 \item[ii)]  for all $m\neq n$, $f^n(D)\cap f^m(D)=\emptyset$, and $D$ is called a \emph{wandering domain}.
\end{itemize}

Suppose that $\{D_0, \cdots, D_{p-1}\}$ is a periodic cycle of Fatou components, then either:
\begin{enumerate}
\item The cycle is (super)attracting: each $D_i$ contains a point of a periodic cycle with multiplier $|\rho|<1$ and  all points in each domain $D_i$ are  attracted to this cycle.
If $\rho=0$, the critical point itself belongs to the periodic cycle and the domain is called super-attracting.
\item The cycle is parabolic: the boundary of each $D_i$ contains a point of a periodic cycle with multiplier $\rho=e^{2\pi i q/m}$, $(q,m)=1$, $m$ a divisor of $p$,  and all points in each domain $D_i$ are
attracted to this cycle.
\item  The components of the cycle are  Siegel disks:  that is, each $D_i$ contains a point of a periodic cycle with multiplier $\rho=e^{2\pi i \theta}$,  where $\theta$ is irrational, and there is a holomorphic homeomorphism
mapping each $D_i$ to the unit disk $\Delta$, and conjugating the first return map
$f^p$ on $D_i$ to an irrational rotation of $\Delta$. The preimages under this conjugacy
of the circles $|\xi|=r, r<1$, foliate the disks $D_i$ with $f^p$ forward invariant
leaves on which $f^p$ is injective.  
\item  The components of the cycle are Herman rings: each $D_i$ is holomorphically homeomorphic to a standard
annulus and the first return map is conjugate to an irrational rotation of
the annulus by a holomorphic homeomorphism. The preimages under this conjugacy of the circles $|\xi|=r, 1<r<R$, foliate the disks with $f^p$ forward invariant leaves on which $f^p$ is injective. 
\item  $D_i$ is an essentially parabolic (Baker) domain: the boundary of each $D_i$
contains a point $z_i$ (possibly $\infty$) and  $f^{np}(z)\to z_i$ for all $z\in D_i$, but $f^p$ is not
holomorphic at $z_i$. If $p=1$, then $z_0=\infty$.
\end{enumerate}

\medskip
\begin{definition} 
The {\em post-singular set} of $f$ is the closure of the orbits of the singular values plus infinity, that is,
\[ 
P\cals_f=\overline{\cup_{n=0}^{\infty} f^n(\cals_f)}\cup \{\infty\}.  \]
For notational simplicity, if a pre-pole $s\in \calpp$ is a singular value, $\cup_{n=0}^{\infty} f^n(s)$ means the finite set $\cup_{n=0}^{N} f^n(s)$ with $f^{N+1} (s)=\infty$.
\end{definition}

Each non-repelling periodic cycle is associated to some singular point.  
In particular we have, see e.g.~\cite[chap 8-11]{Mil} or~\cite[Sect.4.3]{Berg}.

\medskip
\begin{theorem} If $\{D_0, \cdots, D_{p-1}\}$ is an attracting, superattracting, parabolic, or Baker periodic cycle of Fatou components, then for some $i =0, 1, \ldots, p-1$, $D_i$ contains a singular value. If $\{D_0, \cdots, D_{p-1}\}$  is a cycle of rotation domains (Siegel disks or Herman rings), the boundary of each $D_i$ contains the accumulation set of some singular value.
\end{theorem}

We know that when $\#(\cals_f)<\infty$, $f$ may have Baker domains but cannot have any wandering domains (see~\cite{Berg,BKL4,DK}).

% \medskip
% \begin{theorem}~\label{NoWD}
% If $\cals_f$ is finite, then there are no wandering domains in the Fatou set.
% \end{theorem}

\section{Dynamically Natural slices of Parameter spaces}\label{sec:dynamically natural}

\subsection{Holomorphic Families}

We now turn from discussing the dynamics of a particular function $f$ in $\calm$ to a discussion of what happens to the dynamics as we perturb $f$.  The definitions below apply to families of functions over a general complex manifold $X$.  Since the family $\calm$ depends holomorphically on finitely many parameters, \cite{DK},  it can be thought of as a complex manifold. We state the definitions below using the notation  $\calm$ for the complex manifold and    denote the points in $\calm$ by $\la$.

\medskip
 \begin{definition}[Holomorphic family]
A {\em holomorphic family} $\calf$ of meromorphic maps $f$ over a complex manifold $\calm$ is a map
$f(\la,z): \calm \times \CC \rightarrow \hat\CC$, such that  for all $\la \in \calm$,  $\la  \mapsto f_\la (z)$ is holomorphic for all $z\in \CC$.
\end{definition}

The following equivalence can be found in~\cite{McM} for rational functions and in~\cite{KK1} for transcendental meromorphic functions.

\medskip
\begin{theorem} \label{jstable}
Let $\calf$ be a holomorphic family of meromorphic maps with finitely many singularities, over a complex manifold $\calm$. Suppose $\la_0\in \calm$. Then the following are equivalent.
\begin{itemize}
\item[a)] The number of attracting cycles of $f_\la$ is locally constant in a neighborhood of  $\la_0$.
\item[b)] If  for $i=1, \ldots,N$, $s_i(\la)$  are holomorphic maps parameterizing the singular values of $f_\la$, then the  functions $\la \mapsto f_\la^n(s_i(\la))$ form a normal family on a neighborhood of $\la_0$.
\end{itemize}
\end{theorem}

\medskip
\begin{definition}[$J-$stability]
A parameter $\la_0\in \calm$ is a {\em $J$-stable parameter }for the family $\calf$ if it satisfies either of the above conditions.
\end{definition}

In essence, at the $J$-stable parameters of a family, any dynamical behavior, for example, entropy, is a local constant. On the complementary set, the {\em bifurcation locus of the family}, the dynamic behavior,  bifurcates in the neighborhood of every point. There can be points $\la_0$ in the bifurcation locus where only some of the singular values bifurcate. That is, for some but not all values of $i$, the families $\{g_n^i(\la):= f_\la^n(s_i(\la))\}$ are normal in a neighborhood of $\la_0$. If for some $i$, $\{g_n^i(\la)\}$ is not normal in a neighborhood of $\la_0$, then $\la_0$ is said to be in the bifurcation locus of $s_i(\la)$.

\subsection{Dynamically Natural slices}\label{subsec:dynamically natural}
 
 In \cite{FK}, {\em dynamically natural} complex one-dimensional slices of a holomorphic family in $\calm$ were studied. In these slices,  roughly speaking, all the dynamic phenomena but one are fixed, and the last is determined by a ``free asymptotic value''. The components of the complement of the bifurcation locus in these slices contain  $J$-stable parameters. The full definition is as follows (see~\cite{FK}).

\medskip
 \begin{definition}\label{dynnatslice}
Let $\calf=\{ f_{\la} \}_{\la\in \Lambda}$ be a family parametrized by a  one-dimensional complex subspace $\Lambda$ in $\calm$.   We say  $\Lambda$
 is a {\em dynamically natural slice} for the  family $\calf$ in $\calm$   if the following conditions are satisfied:
\begin{itemize}
\item[a)] $\Lambda$ is holomorphically isomorphic to the complex plane punctured at points called {\em parameter singularities} where the function is not in the family; for example, points where the number of singular values is reduced. For readability, we do not distinguish between $\Lambda$ and its image in $\CC$ and denote either by $\Lambda$.  We use $\la$ to denote points in $\Lambda$ and $f_\la$ to denote functions in $\calf$. The singular values of $f_\la$ are distinct holomorphic functions denoted by $s_{i}(\la)$, $i=1,\ldots, N-1$, and a distinguished asymptotic value called the {\em free asymptotic value} that we denote by $v_{\la}$.
\item[ii)] The poles of $f_\la\in \calp_\la$ are distinct holomorphic functions of $\la$; they are denoted by $p_{\la,i}$, $i\in\ZZ$.
\item[ii)] The critical values and some of the asymptotic values are attracted to an attracting or parabolic cycle\footnote{If the cycle is parabolic, the functions in the slice are in the bifurcation locus of that singular value.  See~\cite{FK} for further discussion}, whose period and multiplier are constant for all $\la\in \Lambda$. We call these the {\em dynamically fixed singular values}.  The remaining singular values include $v_\la$; if there is only one remaining value, it may have arbitrary behavior. If there are several,  we require that they  satisfy a relation that persists throughout $\Lambda$ so that the remaining dynamical behavior is controlled by the behavior of $v_\la$. (See \cite{CK,CJKBif,FK} for examples where this happens.) 
\item[iv)] Suppose $v_{\la_0}$ is attracted to an attracting cycle $O_{\la_{0}}$ that does not attract any of the dynamically fixed singular values.  Then the slice $\Lambda$ contains, up to affine conjugacy,  all meromorphic maps $g:\CC\to \widehat{\CC}$ that are quasiconformally conjugate to $f_{\la_0}$ in $\CC$ and conformally conjugate to $f_{\la_0}$ on $\CC \setminus \cala_{\la_0}$, where $\cala_{\la_0}$ is the basin of attraction of the cycle $O_{\la_0}$. 
\item[v)] $\Lambda$ is maximal in the sense that if 
$\Lambda'= \Lambda \cup \{\la_0\}$ where $\la_0$ is a parameter singularity, then $\Lambda'$ does not satisfy at least one of the conditions above.
\end{itemize}
\end{definition}

In a hyperbolic component of the complement of the bifurcation locus of $v_\la$ in this slice,  $v_\la$ is attracted to an attracting cycle $O_{\la}$ of fixed period, called {\em the period of the component}. There are two possibilities:

\medskip
\begin{definition}~\label{shell}
\begin{enumerate}
\item $O_{\la}$ does not attract any of the dynamically fixed singular values. Such components are called {\em Shell components}. We denote such a component by $\Omega$.
\item $O_{\la}$ attracts one of the dynamically fixed singular values in addition to attracting $v_\la$.  These are called {\em capture components}. We do not consider them in this paper.
\end{enumerate}
\end{definition}

\medskip
\subsection{Shell Components}~\label{sec:shell components}
 Our interests here in transversality take place at the boundary points of shell components in dynamically natural slices in $\calm$. They were studied in detail   in~\cite{FK}.  In this section, we give relevant definitions and state the properties we need for our main results.  
 
As above,  $\Lambda$  is a dynamically natural slice for the famliy $\calf$ in $\calm$.  
 
 The set of poles of $f_\la$,  $\calp_{\la}$, is countable; we write
 $$ 
\calp_{\la}=\{ p_{\la,i}\;|\; i \in \ZZ\}.
$$
We reserve the notation  $v_\la$ for the free asymptotic value of the slice.  We denote a shell component in the complement of the bifurcation locus in $\Lambda$ by $\Omega$ or $\Omega_m$ if we want to emphasize the period of the attracting cycle  $O_\la=\{z_{\la,0}, z_{\la,1}, \ldots, z_{\la,m-2},  z_{\la,m-1} \} $ that attracts $v_{\la}$. 

 We need the following definition for special parameters in the bifurcation locus of $\Lambda$. 

\medskip
\begin{definition}~\label{vc} If, for  $\la^{*} \in \Lambda$, there is a  pole $p_{\la^{*}}$  of $f_{\la^{*}}$  and an integer $m>1$ such that
$f_{\la^{*}}^{m-2}(v_{\la^{*}} )= p_{\la^{*}}$, then $\la^{*}$ is called a {\em virtual cycle parameter of order $m$},  $v_{\la^{*}}$ is called a {\em virtual periodic point} and 
$$
O(\la^{*})=\{z_0(\la^{*})=v(\la^{*}), z_1(\la^{*}), \ldots, z_{m-2}(\la^{*})=p(\la^{*}),  z_{m-1}(\la^{*})=\infty\}
$$ 
%$$
%O_{\la^{*}}=\{z_{\la^{*},0}=v_{\la^{*}}, z_{\la^{*},1}, \ldots, z_{\la^{*},m-2}=p_{\la^{*}},  z_{\la^{*},m-1}=\infty\}
%$$ 
is called a virtual cycle (of period $m$). 
\end{definition}

We next define special parameters on the boundary of a shell component $\Omega$. 

\medskip
\begin{definition}~\label{vcp}    
Suppose that as $n \to \infty$,  $\la_n\in \Omega \to \la^{*} \in \partial{\Omega}$ and the multiplier $\rho_{\la_n}=\rho(O_{\la_n})=\Pi_{n=0}^{m-1} f'(z_{\la,i}^n) \to 0.$   Then $\la^{*}$ is called a {\em virtual center} of $\Omega$.
\end{definition}
%Since the attracting basin of the cycle $\mathbf a_\la$ must contain $v_\la$, we will assume throughout that the points in the cycle are labeled so that $v_\la$ and $a_{\la,1}$ are in the same component of the immediate basin of the cycle.

The following theorem summarizes the main results in~\cite{FK} about shell components in a dynamically natural slice $\Lambda$.  Let 
$$
\DD = \{ \xi \in \CC\;|\; 0<|\xi| <1 \}
$$ 
be the punctured disk. 

\medskip
\begin{theorem}~\label{shellcomponents}
Let $\Omega$ be a shell component in $\Lambda$. 
\begin{enumerate}
\item The multiplier map $\rho_\la: \Omega \rightarrow \DD$ is a universal covering map.  It extends continuously to $\partial\Omega$, and $\partial\Omega$ is piecewise analytic;  either $\Omega$ is simply connected, and $\mu_\la$ is infinite-to-one, or $\Omega$ is isomorphic to $\DD$, and the point corresponding to the puncture is a parameter singularity.
\item There is a unique virtual center $\la^{*}$ on $\partial\Omega$. If the period of the component is $1$, $v_{\la^{*}}=\infty$ and constitutes a virtual cycle consisting of a single point. 
\item  Every finite virtual center of a shell component is a virtual cycle parameter, and any virtual cycle parameter on the boundary of a shell component is the virtual center of this shell component.
\end{enumerate}
\end{theorem}

\section{Dynamics around A Virtual Center}~\label{sec:tran}
%Continuing with our notation above, $\calf = \{ f_{\la} \;|\; \la \in \Lambda\}$) is a dynamically natural slice in $\calm$. 
If $\la^{*} \in \Lambda$ is a virtual center of a shell component $\Omega$ of order $m >1$,  the free asymptotic value $v_{\la^{*}}$ is a virtual periodic point of order $m$. The virtual cycle is  
 $$
 \{v_{\la^{*}}, f_{\la^{*}}(v_{\la^{*}}), f_{\la^{*}}^2(v_{\la^{*}}), \ldots, f_{\la^{*}}^{m-3}(v_{\la^{*}}), p_{\la^{*}}, \infty\}.
 $$
Label the points $c_{0}=p_{\la^{*}}$, $c_{1}=\infty$, $c_2=v_{\la^{*}}$ and $c_{i} = f_{\la^{*}}^{i-2} (v_{\la^{*}})$ for $3 \leq i\leq m-1$ so that  
   $f_{\la^{*}}^{m-2} (c_{2}) =c_{0}$. For $\lambda \in \overline{\Lambda}$, consider the function 
\begin{equation}~\label{tran}
\Phi (\la)=f_{\la}^{m-2}(v_\la).
\end{equation}
Clearly, the function is smooth. 
The equation 
\begin{equation}~\label{traneq}
\Phi (\lambda) -c_{0}=0
\end{equation}
has a solution $\la^{*}$.  

\begin{definition}\label{transv} We say $\Phi(\la)$ is  {\em transversal at $\la^{*}$} or that {\em transversality holds at $\la^{*}$} if the derivative $\Phi' (\la^{*}) \not=0$;  this implies that $\la^{*}$ is a simple solution of Equation (\ref{traneq}). 
\end{definition}

The purpose of this section is to prove that in any dynamically natural slice $\Lambda$, transversality holds at every virtual center. More precisely, we will prove:

\medskip
\begin{starthm1}~\label{transversality}
Suppose $\Lambda$ is a dynamically natural slice  for the family $\calf =\{ f_{\la} \;|\; \la\in \Lambda\}$ in $\calm$. Then, transversality holds at every virtual center $\la^{*} \in \Lambda$.  
\end{starthm1}

We divide the proof into several subsections. 

 \subsection{ Holomorphic motions, lifts of holomorphic motions and the transfer operator}

Since our proof of transversality uses holomorphic motions, lifts of these motions, and the transfer operator,  we recall their definitions before we apply them in our context. For a comprehensive study of holomorphic motions, we refer the reader to~\cite{GJW}. 
Let
$$
\Delta_{r}=\{ s \in \CC\; |\; |s| < r\}
$$
be the disk of radius $r>0$ centered at $0$ and let $\Delta=\Delta_{1}$ be the unit disk. 
 
\medskip
\begin{definition}\label{holommot}
Suppose $P$ is a subset of the Riemann sphere $\widehat{\CC}$.  
A map $h (s, z): \Delta_{r}\times P\to \widehat{\mathbb{C}}$ is called a {\em holomorphic motion of $P$ over $\Delta_{r}$}  if 
\begin{itemize}
\item[(1)] $h(0,z)=z$ for all $z\in P$;
\item[(2)] for any fixed $s \in \Delta_{r}$, the map $h_{s} (\cdot)=h(s, \cdot): P\to \widehat{\mathbb{C}}$ is injective;
\item[(3)] for any fixed $z\in P$, the map $h_{z} (\cdot)=h(\cdot, z): \Delta_{r}\to  \widehat{\mathbb{C}}$ is holomorphic.
\end{itemize}
\end{definition}

We want to define holomorphic motions for the family of meromorphic functions  in a given dynamically natural slice $\Lambda$. To this end, 
let $\la^{*}\in \Lambda$ be a virtual center of the slice and let $P$ be the subset of the Riemann sphere defined by the virtual cycle of $f_{\la^{*}}$,  
\[P=O_\la^{*}=\{ z_{\la^{*},0}=p_{\la^{*}}, z_{\la^{*},1}=\infty, z_{\la^{*},2}=v_{\la^{*}}, \ldots,z_{\la^*,m-1 }\}.\]

%\[   
%P=\{ c_{0}, c_{1}, c_{2}, \ldots, c_{m-1}\} . 
%\]  

Since $\la^{*}$ is fixed, for notational simplicity in the formulas below,   we set $g=f_{\la^{*}}$ and $c_i=z_{\la^*,i}$.  Then  $c_{0}$ is a pole of $g$, $c_1=\infty=g(c_0)$, $c_2=g(\infty)$ as a directional limit, $c_{i+1}=g(c_{i})$ for $i= 2, \ldots, m-2$, and $g(c_{m-1})=c_m=c_{0}$. 

The discussion below is valid if we use the parameter $w=v_{\la}$ instead of $\la$ to index the functions in  the dynamical slice.   We therefore set $F_w(z)=F(w,z)=f_{\la}(z)$ 
and  denote the set of poles and critical points of $F_{w}$ by $\cale_{w}$; we denote the finite singular set of  $F_w$ consisting of critical values and asymptotic values by $\tilde{\cals}_w$. Define the Riemann surfaces 
\[
S_{w,0}=\mathbb{C}\setminus \cale_{w} \quad \hbox{and}\quad S_{w,1}=\mathbb{C}\setminus \tilde{\cals}_w.
\]
Since  $w=v_{\la}$ is an asymptotic value of the complex map $F_{w}(z)$, it follows that

\medskip
\begin{lemma}\label{cov}
The map 
\[
F_{w}(z) : S_{w,0}\to S_{w,1}
\]
is a holomorphic covering map of infinite degree.
\end{lemma}

\medskip
In the following, to make our notation agree with that in Definition~\ref{holommot},  we set $s=\la-\la^{*}$ and let $\Delta$ be the disk centered at $s=0$;  as above, we let  $P=\{ c_{0}, c_{1}, c_{2}, \ldots, c_{m-1}\}$  be the virtual cycle of $f_{\la^{*}}$.  

\medskip
\begin{definition}~\label{liftdef}
Suppose $h (s, z): \Delta\times P\to \widehat{\mathbb{C}}$ is a holomorphic motion with $h(s, c_{1})=c_{1}=\infty$ for all $s\in \Delta$; set $c_{2}(s) =h(s, c_{2})$.  We say another holomorphic motion $\widehat{h} (s, z): \Delta\times P\to \widehat{\mathbb{C}}$ is a {\em lift of $h$} if 
\begin{equation}~\label{lifteq}
h_{s}(g(c_{i})) = F_{c_{2}(s)}(\widehat{h} (s, c_i)),
\end{equation}
for all $ i=0,1, \ldots, m-1$ and all $s\in \Delta$; that is, the following diagram commutes
$$
\xymatrix{{\Delta \times P} \ar[r]^{\widehat{h}}\ar[d]_{Id \times g} &\CC\ar[d]^{F_{c_{2}(s)}}\\
\Delta \times P\ar[r]^{h}& \CC.}
$$
\end{definition}

Let $c_{i}(s) =h(s, c_{i})$ for $0\leq i\leq m-1$. 
From (\ref{lifteq}), $\widehat{h}(s, c_{0}) =p_{c_{2}(s)}$ is a pole of $F_{c_{2}(s)}$. Since $c_{2}(s)=w=v_{\la}$ is the free asymptotic value of $F_{c_{2}(s)}=F_{v(\la)}$, from (\ref{lifteq}), $\widehat{h} (s, c_{1})=c_{1}=\infty$. 

\medskip
\begin{lemma}~\label{lift}
Suppose $h (s, z): \Delta\times P\to \widehat{\mathbb{C}}$ 
is a holomorphic motion such that $h(s,c_{1})=c_{1}$ for all $s\in \Delta$. 
Then we can find a lift of $h$.
\end{lemma}

\begin{proof}
First, define $\widehat{h}(s, c_{0}) =p_{c_{2}(s)}$ and $\widehat{h} (s, c_{1}) =c_{1}$ for all $s\in \Delta$. Since $c_{m}=c_{0}$,
$$
F_{c_{2}(s)}:\cals_{c_{2}(s), 0}\to \cals_{c_{2}(s),1}
$$
is a covering map, $h(s, c_{i+1}): \Delta\to \cals_{c_{2}(s),1}$ for each $i, \, 2\leq i\leq m-1$ is a holomorphic map,  $\Delta$ is simply connected, and $F_{c_{2}(0)} =g$.  Therefore we can lift it to a holomorphic map $\widehat{h} (s, c_{i}): \Delta \to \cals_{c_{2}(s), 0}$ with $\widehat{h}(0, c_{i}) =c_{i}$; that is, the following diagram commutes,  
$$
\xymatrixcolsep{5pc}
\xymatrix{&\cals_{c_{2}(s), 0} \ar[d]^{F_{c_{2}}(s)}\\
\Delta \ar[r]^{h(s, c_{i+1})}  \ar[ru]^{\widehat{h} (s, c_{i})} & \cals_{c_{2}(s),1}}
$$
Thus we have a map $\widehat{h} (s, z): \Delta\times P\to \widehat{\CC}$ such that for each fixed $z=c_{i}$, $0\leq i\leq m-1$, $\widehat{h} (s, c_{i}): \Delta\to \widehat{\CC}$ is holomorphic and $\widehat{h} (0, c_{i}) =c_{i}$, $0\leq i\leq m-1$. 

To show $\widehat{h} (s, z)$ is a holomorphic motion, we now need only show that for every fixed $s\in \Delta$, $\widehat{h} (s, \cdot): P\to \widehat{\CC}$ is injective.  Since $h(s, c): \Delta\times P\to \widehat{\CC}$ is a holomorphic motion, if $i \neq j$, $c_{i}\not=c_{j}$. Because $\widehat{h} (s, c_{i})$ and $\widehat{h} (s, c_{j})$ are lifts of $h (s, c_{i+1})$ and $h (s, c_{j+1})$ under the covering map $F_{c_{2}(s)}: \cals_{c_{2}(s), 0}\to \cals_{c_{2}(s), 1}$, $h (s, c_{i+1})\not=h (s, c_{j+1})$.  Thus, $\widehat{h} (s, \cdot)$ is injective.  

%$\widehat{h} (s, c_{i})\not=\widehat{h} (s, c_{j})$  and   for $0\leq i, j\leq m-1$, if $i \neq j$,  $c_{i}\not=c_{j}$,  $\widehat{h} (s, c_{i})$ and $\widehat{h} (s, c_{j})$ are lifts of $h (s, c_{i+1})$ and $h (s, c_{j+1})$ under the covering map $F_{c_{2}(s)}: \cals_{c_{2}(s), 0}\to \cals_{c_{2}(s), 1}$ and  
% (. This completes the proof. 
\end{proof}

Suppose $\widehat{h}$ is a lift of the holomorphic motion $h$. If we differentiate  equation (\ref{lifteq})
% $$
% h (s, g(c_{i}))= F_{c_{2}(s)} (\widehat{h} (s, c_{i})).
% $$
at $s=0$ for each $i=2, \ldots, m-1$, we get an $(m-2)\times (m-2)$-matrix $A$ such that 
$$
%\begin{equation}~\label{tranop}
\Big( \frac{\partial\widehat{h}(s, c_{i})}{\partial s}\Big|_{s=0}\Big)=A\Big( \frac{\partial h(s, c_{i})}{\partial s}\Big|_{s=0}\Big),
%\end{equation}
$$
where $\Big(\frac{\partial\widehat{h}(s, c_{i})}{\partial s}\Big|_{s=0}\Big)$ and $\Big( \frac{\partial h(s, c_{i})}{\partial s}\Big|_{s=0}\Big)$ denote two column vectors in $\RR^{m-2}$. The entries in the matrix $A$ depend only on the partial derivatives of $f(s,z)$  at $s=0$ evaluated at the points of $P$.   
Computing,   we obtain the matrix $A=(a_{i,j})_{(m-2)\times(m-2)}$, where
$$
a_{i,2} = \frac{- \frac{\partial F}{\partial w}(c_2, c_i)}{\frac{\partial F}{\partial z}(c_2, c_i)}, \, \, i = 2, \ldots, m-1,
$$
$$
a_{i,i+1} = \frac{1}{\frac{\partial F}{\partial z}(c_2, c_i)}, \, \, i = 2, \ldots, m-1,
$$
%a_{1,2} = \frac{1-\frac{\partial F}{\partial w}(c_2, c_1)}{\frac{\partial F}{\partial z}(c_2, c_1)},\\
$$
a_{i,j} = 0 \quad \mbox{otherwise}.
$$
That is
$$
A=\begin{pmatrix}~\label{form}
\frac{- \frac{\partial F}{\partial w}(c_2, c_2)}{\frac{\partial F}{\partial z}(c_2, c_2)} & \frac{1}{\frac{\partial F}{\partial z}(c_2, c_2)}&0&0&\cdots &0\\
\frac{- \frac{\partial F}{\partial w}(c_2, c_3)}{\frac{\partial F}{\partial z}(c_2, c_3)} &0& \frac{1}{\frac{\partial F}{\partial z}(c_2, c_3)}&0&\cdots &0\\
\vdots & \vdots &\vdots &\vdots&\vdots &\vdots\\
\frac{- \frac{\partial F}{\partial w}(c_2, c_{m-1})}{\frac{\partial F}{\partial z}(c_2, c_{m-1})} &0&0&0 &\cdots &\frac{1}{\frac{\partial F}{\partial z}(c_2, c_{m-1})}
\end{pmatrix}
$$

\medskip
\begin{definition}\label{transop} We call the matrix $A$ the {\em transfer operator} associated with $g=f_{\la^{*}}$. 
\end{definition}

Given a holomorphic motion $h_{0} (s, z): \Delta\times P\to \widehat{\mathbb{C}}$ such that $h_{0}(s,c_{1})=c_{1}$, we can apply Lemma~\ref{lift} to  construct a sequence of holomorphic motions
$$
\{ h_{k} (s, z): \Delta\times P\to \widehat{\CC} \} _{k=0}^{\infty}
$$ 
such that $h_{k+1}$ is a lift of $h_{k}$. Let $\widehat{P}=P\setminus \{ c_{1}\}$. 
 
\medskip
\begin{lemma}~\label{lift1}
There is an $r, \, 0< r\leq 1$ such that 
$$
\sup_{s\in \Delta_{r}, z\in \widehat{P}, k=0, 1, \ldots} |h_{k} (s, z)| < \infty.
$$
\end{lemma} 

\begin{proof}
% We prove the lemma by contradiction. Suppose that for every $r<1$, 
% $$
% \sup_{s\in \Delta_{r}, z\in \widehat{P}, k=0, 1, \cdots} |h_{k} (s, z)| = \infty.
% $$  
% Since $\widehat{P}$ contains only finitely many points, this assumption implies that for some fixed $x\in \widehat{P}$ there is a sequence $\{k_{n}\}$ of integers tending to $\infty$ and a sequence of complex numbers $s_{n}\in \Delta$ tending to $0$ such that $h_{k_{n}}(s_{n}, x) \to \infty$ as $n\to \infty$. 

For each $k\geq 1$, $c_{0,k}(s) =h_{k} (s, c_{0})=p_{c_{2,k}(s)}$ is the pole of $F_{c_{2, k}(s)}$, $c_{1,k}(s)= h_{k} (s, c_{1}) =\infty$, and $c_{2,k}(s) =h_{k} (s, c_{2})=v_{c_{2,k}(s)}$ is the free asymptotic value of $F_{c_{2,k}(s)}$. At $s=0$, $c_{0,k}(0) =p_{\la^{*}}$ and $c_{2,k}(0) =v_{\la^{*}}$. Since the holomorphic motions  are injective, we can find an $r_{0}\in (0, 1)$ and a constant $M_{0}>0$ such that for all $s\in \Delta_{r_{0}}$,
$$
|c_{2,k}(s)-c_{0,k}(s)|\geq \frac{1}{M_{0}}\quad \hbox{and} \quad |c_{0,k}(s)|, |c_{2,k}(s)| \leq M_{0}.
$$

For every $c_{i}$, $3\leq i\leq m-1$,  at every $s\in \Delta_{r_{0}}$, the holomorphic function
$$
h_{k} (s, c_{i}) \not = c_{0,k} (s), c_{1,k}(s),  c_{2, k}(s).
$$
Thus, the family 
$$
\{ h_{k}(s, c_{i}) \;|\; s\in \Delta_{r_{0}}\}
$$ 
is a normal family. This implies that the sequence $h_{k_{n}} (s, c_{i})$ has a subsequence, which we again denote by $h_{k_n}(s, c_{i})$, that converges to a holomorphic function $h_{i}(s)$ or to the constant $\infty$. Since $h_{k_{n}} (0,c_{i})=c_{i}$ for all $n$, $h_{k_{n}} (s, c_{i})$ must converge to a holomorphic function $h_{i}(s)$ on $\Delta_{r_{0}}$;  therefore, there are constants $0<r_{i}<1$, $0<M_{i}<\infty$, such that for every $n\geq 1$, $|h_{k_{n}} (s, c_{i})|\leq M_{i}$ for all $s\in \Delta_{r_{i}}$. Let 
$$
r=\min\{r_{i}\;|\; i=0, 2, 3, \cdots, m-1\}.
$$
Then we have
$$
\sup_{s\in \Delta_{r}, z\in \widehat{P}, k=0, 1, \ldots} |h_{k} (s, z)| < \infty.
$$
\end{proof}    

\subsection{The Spectral Radius of the Transfer Operator.}

We state the following result as a lemma; it is another important component in our proof of transversality. Let
$$
a=\min\{ |c_{i}-c_{j}| \;|\; i,j \neq 1, 0\leq i\not=j\leq m-1\}.   
$$

\medskip
\begin{lemma}~\label{sp}
The spectral radius of $A$ is less than or equal to $1$.
\end{lemma}

\begin{proof} 
 Let $\mathbf{v_0}=(v_{0,i})$ be a column vector in $\mathbb{R}^{m-2}$
satisfying $|v_{0,i}| \leq a/3$ for all $0\leq i\leq m-1$. Let $v_{0,0}=v_{0,1}=0$.
Define a holomorphic motion of $P$ over $\Delta$ by $h_{0} (s, c_i)=c_{i}+s v_{0,i}$, $0\leq i\leq m-1$. The bound on $|v_{0,i}| $ ensures injectivity. Note that $h_{0}(s, c_{1})=c_{1}$ for all $s\in \Delta$. 

Applying Lemma~\ref{lift}, we can construct a sequence of holomorphic motions 
$$
\{ h_{k} (s, z): \Delta\times P\to \widehat{\CC} \} _{k=0}^{\infty},
$$ 
such that  $h_{k+1}$ is a lift of $h_{k}$. 
Applying Lemma~\ref{lift1}, we have two constants $0<r<1$ and $0<M<\infty$ such that 
\begin{equation}~\label{bound}
\sup_{s\in \Delta_{r}, z\in \widehat{P}, k=0, 1, \ldots} |h_{k} (s, z)| < M.
\end{equation}

By definition, $v_{0,i}=\frac{\partial h_0(s,c_i)}{\partial s}|_{s=0}$.  Repeatedly  applying the transfer operator $A$,  we obtain a sequence of vectors  
$$
{\mathbf v}_{k+1}=A {\mathbf v}_k = A^k \mathbf{v_0}.
$$
From (\ref{bound}) and the Cauchy-Schwarz inequality, there is a constant, which we will still denote as $M>0$, such that 
$$
| v_{k,i}|\leq M, \quad \forall\; k\geq 0, \;\; i=2, \cdots, m-1. 
$$
Hence
\[ \| A^k {\mathbf v}_0 \| \leq M, \quad \forall\; k>0.  \]
This implies that $\|A^{k}\| \leq (3M)/a$ so that  the spectral radius 
$$
\rho=\lim_{k\to \infty} \sqrt[k]{\|A^{k}\|} \leq   1.
$$
 \end{proof}
 
 \subsection{Transversality and the Leading Eigenvalue of $A$}

Since the spectral radius of $A$ is less than or equal to $1$, the norms of all the eigenvalues of $A$ are less than or equal to $1$. 
In this section, we show that if  $1$ is not an eigenvalue of $A$ is then $A$ transversal; that is,  if $\Phi (\la)$ is the function defined in (\ref{tran}) then transversality of $A$ holds if and only if $\Phi'(\la^{*})\not=0$.

\medskip
\begin{starthm2}~\label{eig1}
$\Phi' (\la^{*}) =0$ if and only if $1$ is   an eigenvalue of $A$.
\end{starthm2}

\begin{proof}
We prove the ``if'' statement first.   Suppose $1$ is an eigenvalue of $A$. This means that $A$ has a (non-zero) column eigenvector ${\mathbf v} = (v_{i})$ in $\mathbb{R}^{m-2}$ with $\max_{2\leq i\leq m-1} |v_i | =a/3$ that satisfies $A{\mathbf v}={\mathbf v}$. Let $v_{0}=v_{1}=0$. Note that $v_{2}\not=0$; otherwise, from the form of $A$,  ${\mathbf v}={\mathbf 0}$. 

Define a holomorphic motion of $P$ over $\Delta$ by $h(s,c_{i}) =c_{i}+s v_{i}$ for $0\leq i\leq m-1$. Suppose $\widehat{h}(s,c_{i})$ is a lift of $h$. Since $A{\mathbf v}={\mathbf v}$,  we have  
$$
\widehat{h}(s,c_{i})=h(s, c_{i}) +{\rm O}(|s|^{2})= c_{i}+s v_{i}+{\rm O}(|s|^{2}), i=2,3,\cdots, m-1.
$$
From Equation (\ref{lifteq}), we obtain
$$
F_{c_{2}(s)} (\widehat{h}(s, c_{i})) = h(s, c_{i+1}) +{\rm O}(|s|^{2}),\;\; i=2, \ldots, m-2, 
$$
or equivalently, 
$$
F_{c_{2}(s)} (c_{i}+sv_{i}) = c_{i+1}+sv_{i+1} +{\rm O}(|s|^{2}), \;\; i=2, \ldots, m-2.
$$
For $i=m-1$, we have
$$
F_{c_{2}(s)} (h (s, c_{m-1})) = h(s, c_{0}) +{\rm O}(|s|^{2}),
$$ 
or equivalently, 
$$
F_{c_{2}(s)} (c_{m-1}+sv_{m-1}) = c_{0} +{\rm O}(|s|^{2}).
$$ 
From the above, we conclude that 
$$
 F_{c_{2}(s)}^{m-2} (c_{2}(s))-  c_{0}  = {\rm O}(|s|^{2}).
$$
Since $c_{2}(s)=v_{\la(s)}$ and $c_{2}(0)=v_{\la^{*}}$, set $\phi(c_{2}(s))=\Phi(\la)$. Then, 
$$
\phi' (c_{2}(0)) c_{2}'(0)= \phi' (c_{2}(0)) v_{2}=0
$$
which implies $\Phi'(\la^{*})=\phi' (c_{2}(0))=0$ as claimed.
 \medskip

The proof of the ``only if'' statement is easier. Since we assume  $\Phi'(\la^{*}) =\phi' (c_{2}(0))=0$, there is  a $v_{2}\not=0$ such that $|v_{2}|\leq a/3$, satisfying $\phi (c_{2}+sv_{2}) ={\rm O}(|s|^{2})$.  Let $c_{2}(s) =c_{2} +sv_{2}$. Define a holomorphic motion $h(s,z): \Delta\times P\to \widehat{\CC}$ such that $h(s,c_{2}) =c_{2}(s)$ and $h(s,c_{1})=c_{1}$. Let $c_{i} (s) =h(s, c_{i})$ for $i=3, \cdots, m-1$ and define  
$$
v_{i} = \frac{d F_{c_{2}(s)} (c_{i}(s))}{ds} \big|_{s=0}.
$$ 
 The non-zero column  vector ${\mathbf v} =(v_{i})\in \mathbb{R}^{m-2}$ satisfies $A{\mathbf v}={\mathbf v}$ so that $1$ is an eigenvalue of $A$.  
\end{proof} 

\subsection{Asymptotic invariance of lifts of holomorphic motions.}
The asymptotic order of a holomorphic motion measures how close it is to the identity.  The next definition is used in Lemma~\ref{liftorder} below to relate the asymptotic order of a motion and its lift. 

\medskip
\begin{definition}
A holomorphic motion $h$ of $P$ over $\Delta_{r}$  is called  {\em asymptotically invariant of order $l$} if there is a lift $\widehat{h}$ such that
$$
\widehat{h}(s,z)-h(s,z)= {\rm O}(|s|^{l+1}), \quad \forall \; s\in \Delta_{r}, \; z\in P.
$$
\end{definition}

The following lemma is standard calculus.  We write out the proof to establish the notation for the proof of the following lemma. 

\begin{lemma}\label{diff}
Suppose $F(w, z): \mathbb{C}^2\to \mathbb{C}$ is holomorphic and $\sigma(s), \tau(s): \Delta_{r}\to \mathbb{C}^2$  are two holomorphic maps such that $\sigma(0)=\tau(0)=(a,b)$. Suppose further that for some $l>0$, 
$$
\sigma(s)-\tau(s) ={\rm O}(|s|^{l}).
$$
Then, writing $\sigma(s)=(\sigma_1(s), \sigma_2(s))$ and $\tau(s)=(\tau_1(s), \tau_2(s))$, we have
$$
F(\sigma(s))-F(\tau(s)) =\frac{\partial F}{\partial w}(a,b)(\sigma_1(s)-\tau_1(s))+ \frac{\partial F}{\partial z }(a,b)(\sigma_2(s)-\tau_2(s))+{\rm O}(|s|^{l+1}).
$$
\end{lemma}

\begin{proof}  The Taylor series for $F(w,z)$ is 

\begin{multline*}
F(w,z)-F(a,b)=\frac{\partial F}{\partial w}(a, b) (w-a)+ \frac{\partial F}{\partial z}(a, b) (z-b)+\frac{\partial^2 F}{\partial w^2}(a, b) (w-a)^2\\
+\frac{\partial^2 F}{\partial w\partial z}(a, b) (w-a)(z-b)+\frac{\partial^2 F}{\partial z^2}(a, b) (z-b)^2+ \text{higher order terms.}
\end{multline*}
Substituting    first   $\sigma(s)=(\sigma_1(s),\sigma_2(s))$ and then  $\tau(s)=(\tau_1(s),\tau_2(s))$ for $(w,z)$ in the Taylor series  and subtracting, we obtain 
\begin{multline*}
F(\sigma(s))-F(\tau(s))=\frac{\partial F}{\partial w}(a,b)(\sigma_1(s)-\tau_1(s))+ \frac{\partial F}{\partial z}(a,b)(\sigma_2(s)-\tau_2(s))\\
+\frac{\partial^2 F}{\partial w^2}(a,b) (\sigma_1(s)+\tau_1(s)-2a)(\sigma_1(s)-\tau_1(s))\\
+\frac{\partial^2 F}{\partial w\partial z}(a, b) [(\sigma_1(s)-a)(\sigma_2(s)-\tau_2(s))+(\tau_1(s)-b)(\sigma_1(s)-\tau_1(s))]\\
+\frac{\partial^2 F}{\partial z^2}(a,b) (\sigma_2(s)+\tau_2(s)-2b)(\sigma_2(s)-\tau_2(s))+ \text{higher order terms}\\
=\frac{\partial F}{\partial w}(a,b)(\sigma_1(s)-\tau_1(s))+ \frac{\partial F}{\partial z}(a,b)(\sigma_2(s)-\tau_2(s))+{\rm O}(|s|^{l+1})
\end{multline*}
\end{proof}

 The next lemma says that the asymptotic order of a holomorphic motion is preserved under lifting; that is, if $h$ is a holomorphic motion asymptotically invariant of order $l$, so is its lift $\widehat{h}$.  
   
\begin{lemma}\label{liftorder}
Suppose that for some $l>0$, $h_0$ and $h_1$ are holomorphic motions of $P$ over $\Delta_r$ satisfying $$h_0(s,c_{i})-h_1(s,c_{i})={\rm O}(|s|^l).$$  If $\widehat{h}_0$, $\widehat{h}_1$ are respective lifts of the motions, then $$\widehat{h}_0(s,c_{i})-\widehat{h}_1(s,c_{i})={\rm O}(|s|^l).$$ 
\end{lemma}

\begin{proof} We use superscripts to denote the functions $c_2(s)$ for each of the motions: 
 $c_2^0(s)=h_0(s,c_2)$ and $c_2^1(s)=h_1(s,c_2)$.  
 By hypothesis,  $c_2^0(s)-c_2^1(s)={\rm O}(|s|^l)$. Since $\widehat{h}_0$ and $\widehat{h}_1$ are lifts of $h_0$ and $h_1$, we have $$h_0(s,c_{i+1})=F_{c_2^0(s)}( \widehat{h}_0(s,c_i))\ \text{and}\  h_1(s,c_{i+1})=F_{c_2^1(s)}(\widehat{h}_1(s,c_i)).$$  Subtracting and applying Lemma \ref{diff}  with $F(w,z)=F_w(z)$, we get
\begin{multline*}
h_0(s,c_{i+1})-h_1(s,c_{i+1})=F_{c_2^0(s)}( \widehat{h}_0(s,c_i))-F_{c_2^1(s)}( \widehat{h}_1(s,c_i))\\
=\frac{\partial F}{\partial w}\big|_{s=0}(c_2^0(s)-c_2^1(s))+\frac{\partial F}{\partial z}\big|_{s=0}(\widehat{h}_0(s,c_i)-\widehat{h}_1(s,c_i))+{\rm O}(|s|^{l+1})
\end{multline*}

Therefore $\widehat{h}_0(s,c_i)-\widehat{h}_1(s,c_i)={\rm O}(|s|^{l})$.
\end{proof}
As an immediate corollary, we have

\medskip
\begin{corollary}
Let $\{h_k\}$ be a sequence of lifts of the holomorphic motion $h_0$.   If $h_1- h_0={\rm O}(|s|^{l+1})$, then $h_{k+1}-h_k={\rm O}(|s|^{l+1})$. 
  \end{corollary}
  
In the next lemma, we construct a new holomorphic motion $H$ from the sequence of lifts, $\{h_k\}$, of a given motion $h_0$ of asymptotic order $l$ such that $H$ is asymptotically invariant of order $l+1$.  

\begin{lemma}\label{increaseorder}
Suppose that for some $r>0$ and $l\geq 1$, we have an asymptotically invariant holomorphic motion $h$ of $P$ over $\Delta_{r}$  of order $l$.
 Then we can construct another holomorphic motion $H$ of $P$ over $\Delta_{r'}$ for some $r'>0$ which is asymptotically invariant of order $l+1$.
\end{lemma}

\begin{proof}
Set $h_{0}(s,z)=h(s,z)$. By Lemma~\ref{lift}, we can find a sequence of holomorphic motions $\{h_{k}(s,z)\}_{k=0}^{\infty}$ such that $h_{k+1}$ is a lift of $h_{k}$ satisfying Lemma~\ref{lift1}. Consider the  means
$$
\mu_{k} (s,z) =\frac{1}{k}\sum_{i=0}^{k-1} h_{i}(s,z)\quad \hbox{and}\quad
\nu_{k} (s,z) =\frac{1}{k}\sum_{i=1}^{k} h_{i}(s,z).
$$
They are both uniformly bounded on $\widehat{P}$ and so by a   normal families argument we can find     subsequences of each,
$$
\mu_{k_{n}} (s,z) =\frac{1}{k_{n}}\sum_{i=0}^{k_{n}-1} h_{i}(s,z) \quad \mbox{and}\quad 
\nu_{k_{n}} (s,z) =\frac{1}{k_{n}}\sum_{i=1}^{k_{n}} h_{i}(s,z)   $$
  such that as $k$ goes to infinity,  both subsequences converge to the same holomorphic limit  $H(s,z)$.  
 This limit is defined on $\Delta_{r}\times P$ and is holomorphic in $s$. Since $H(0,z)=z$ and since $P$ contains only a finite number of points, we can find an $r'$,  $0<r'\leq r$, such that for any $s\in \Delta_{r'}$, $H(s, z)$ is injective on $P$. Thus $H(s,z)$ is also a holomorphic motion.
   
To complete the proof we need to show that   $H$ is asymptotically invariant of order $l+1$.  
 Let  $\widehat{\mu}_k$  denote the lift of $\mu_k$ so  that,
$$F_{\mu_k(s,c_2)}(\widehat{\mu}_k (s,z))=\mu_k(s,g(z)).$$
Assume for the time being that the following claim holds: 

{\em Claim:} For each $z\in \widehat{P}=P \setminus \{c_1\}$ and any $k\geq 1$,
$$F_{\mu_k(s,c_2)}(\nu_k(s,z))=\mu_k(s,g(z))+{\rm O}(|s|^{l+2}).$$

With this assumption and 
  the definitions of $\mu_k, \nu_k$ and $F_{\mu_k(s,c_2)}$, we get 
$$
\widehat{\mu}_k(s,z)-\nu_k(s,z)=\rm{O}(|s|^{l+2})
$$ 
for each $k$. Using an argument similar to the proof of Theorem B, and, if necessary,  taking $r'$ smaller, we can assume $\{\widehat{\mu}_k(s,z)\}$ 
is a bounded sequence on $\Delta_{r'}\times \widehat{P}$. Taking a subsequence if necessary, we obtain a holomorphic limit,  $\widehat{\mu}_k(s,z)\to \widehat{H}(s,z)$ on $\Delta_{r'}\times P$; that is,  $\widehat{H}$ is a lift of $H$. 
Moreover,  the sequence $\{\widehat{\mu}_k(s,z)-\nu_k(s,z))\}$  is also a bounded sequence of holomorphic functions on $\Delta_{r'}$ so that 
$$
\widehat{H} (s,z) - H(s,z) =O(|s|^{l+2}).
$$
Therefore, $H(s,z)$ is asymptotically invariant of order $l+1$.
\medskip

To prove the claim, fix $z\in \widehat{P}$ and $k\geq 1$.   By the construction of $h_i(s,z),$ we have 
$$  F_{h_i(s,c_2)}( h_{i+1}(s,z)) =h_i(s,g(z))$$ for every $i\geq 0$. Thus
$$
\mu_k(s,g(z))=\frac{1}{k}\sum_{i=0}^{k-1} F_{h_i(s,c_2)}(h_{i+1}(s,z)).
$$
By  Lemma~\ref{liftorder}, we have $h_{i+1}(s,z)-h_i(s,z)={\rm O}(|s|^{l+1})$ for all $z\in P$ and $i\geq 0$. Therefore all the functions $h_i(s,z)$, $\mu_i(s,z)$, $\nu_i(s,z)$ have the same derivatives up to order $l$ at $s=0$. Applying  Lemma~\ref{diff}, we get  
\[
F(h_{i} (s,c_2),h_{i+1}(s,z))-F(\mu_k(s,c_2),\nu_k(s,z))=
\]
\[
\frac{\partial F}{\partial c_2}(c_2, z)(h_i(s,c_2)-\mu_k(s,c_2))+\frac{\partial F}{\partial z}(c_2, z)(h_{i+1}(s,c_2)-\nu_k(s,z))+{\rm O}(|s|^{l+2}).
\]
Summing from $i=0$ to  $k-1$, we obtain
$$\frac{1}{k}\sum_{i=0}^{k-1}F(h_i(s,c_2),h_{i+1}(s,z))=F_{\mu_k(s,c_2)}( \nu_k(s,z))+{\rm O}(|s|^{l+2}).$$ Thus we have the equality
$$F_{\mu_k(s,c_2)}(\nu_k(s,z))=\mu_k(s,g(z))+{\rm O}(|s|^{l+2}) $$
 which proves the claim. 
\end{proof}

\subsection{Proof of Theorem A}%~\ref{transversality}}

The main point in the proof is to show that $1$ is not an eigenvalue of the transfer operator $A$ for the family $f_\la$. Then Theorem B implies that 
$\Phi'(\la^{*})\not=0$.  We assume $1$ is an eigenvalue of $A$ and obtain a  contradiction. 

The idea of the proof is as follows. We  define a holomorphic motion over $\Delta$ of the virtual cycle of period $m$ of the function $f_{\la^{*}}$.  We then use the lemmas above to construct new holomorphic motions that get closer and closer to the identity; that is, are asymptotically invariant of arbitrarily high order.  This will imply that for every $\la$ in a small neighborhood of $\la^{*}$, $f_{\la}$ has an attracting periodic cycle of period $m$.   This is a contradiction because $\la^{*}$ is on the boundary of a shell component. 

 First, the function  $\Phi(\la)=f^{m-2}_\la(v_{\la})=\phi (v_{\la})$ is non-constant and holomorphic in $\la$ in a neighborhood $U$ of $v_{\la^{*}}=c_2$. 
 Therefore, for some integer $l \geq 1$ we can find  $a_l \neq 0$ such that
\[ \Phi(\la) =\phi(v_\la)= c_{0}+a_l (v_\la-c_2)^l + {\rm O}(|v_\la-c_2|^{l+1}),\;\; v_{\la} \in U. \]

By our assumption, we can find  a column eigenvector ${\mathbf v}=(v_{i})$ in $\RR^{m-2}$ with 
$$
\max_{2\leq i\leq m-1}|v_i| =a/3
$$
such that $A{\mathbf v}={\mathbf v}$.  As in the proof of the ``if'' part of Theorem B,  we define a non-degenerate holomorphic motion $h_0$ of $P$ over $\Delta$ that is asymptotically invariant of order $1$.  By Lemma~\ref{lift} we obtain a uniformly bounded sequence (that is, satisfying Lemma~\ref{lift1}) of lifts,  $\{h_k\}$, which are holomorphic motions of $P$ over $\Delta_{r_{1}}$ for some $0<r_{1}<1$.
By Lemma~\ref{liftorder}, all the $h_k(s, c_i)$ are asymptotically invariant of order $1$.

Set $H_{0}=h_{0}$ and apply Lemma~\ref{increaseorder}  with $h=H_{0}$,   to obtain a new holomorphic motion $H_{1}$ of $P$ over $\Delta_{r_{1}}$ (taking $r_{1}$ smaller if necessary) such that $H_{1}-H_{0}={\rm O}(s^{3})$;  that is, $H_{1}$ is asymptotically invariant of order $2$.

We now set  $h_{0}=H_{1}$ in Lemma~\ref{lift1} and apply it to obtain a new sequence of lifts, $\{h_{k}\}$; again, Lemma~\ref{liftorder} implies that all of these motions are asymptotically invariant of order $3$. They are holomorphic motions of $P$ over $\Delta_{r_{2}}$ for some $0<r_{2}<r_{1}$.    

 Repeating this process for each  integer $N \geq1$, we can find a non-degenerate holomorphic motion  $H_N$ of $P$ over $\Delta_{r_{N}}$, for some $0<r_{N}<r_{N-1}<1$,  that is 
 asymptotically invariant of order $N$.  
 It follows that  
$$ 
\phi( H_{N}(s,c_2)) = F^{N-1}_{H_N(s,c_2)} (H_N(s,c_2) )=c_0 + {\rm O}(|s|^{N+1}). $$ In particular, if $N=l+1$, this implies that $a_l=0$ giving us the required contradiction.

\end{document}